\documentclass{article}
\usepackage[parfill]{parskip}
\usepackage[margin=1.5in]{geometry}
\usepackage{times}
\usepackage{amssymb,amsmath,color}
\usepackage{dsfont}
\usepackage{url}

\usepackage{graphicx}
\usepackage{subfigure}

\usepackage{algorithm}
\usepackage{algorithmic}

\usepackage{hyperref}

\usepackage{amsthm}
\theoremstyle{plain}
\newtheorem{lemma}{Lemma}
\newtheorem{theorem}{Theorem}

\newtheorem{corollary}{Corollary}
\theoremstyle{definition}

\newcommand{\Sec}[1]{Section~\ref{sec:#1}}
\newcommand{\Eq}[1]{Eq.~(\ref{eq:#1})}
\newcommand{\Alg}[1]{Alg.~(\ref{alg:#1})}
\newcommand{\Fig}[1]{Figure~\ref{fig:#1}}

\newcommand{\Theorem}[1]{Theorem~\ref{th:#1}}
\newcommand{\Corollary}[1]{Corollary~\ref{cor:#1}}
\newcommand{\Lemma}[1]{Lemma~\ref{lem:#1}}

\newcommand{\BEAS}{\begin{eqnarray*}}
\newcommand{\EEAS}{\end{eqnarray*}}
\newcommand{\BEA}{\begin{eqnarray}}
\newcommand{\EEA}{\end{eqnarray}}
\newcommand{\BEQ}{\begin{equation}}
\newcommand{\EEQ}{\end{equation}}
\newcommand{\BIT}{\begin{itemize}}
\newcommand{\EIT}{\end{itemize}}
\newcommand{\BNUM}{\begin{enumerate}}
\newcommand{\ENUM}{\end{enumerate}}
\newcommand{\BA}{\begin{array}}
\newcommand{\EA}{\end{array}}

\newcommand{\diag}{\mathop{\rm diag}}

\newcommand{\one}{\mathds{1}}
\newcommand{\scprod}[2]{\langle#1,#2\rangle}

\newcommand{\argmin}{\mathop{\rm argmin}}

\newcommand{\tr}{\mathop{ \rm tr}}

\def \S{  { \Sigma} }

\def \R{{\mathbb R}}
\def \N{{\mathbb N}}

\newcommand{\ie}{i.e.\ }
\newcommand{\eg}{e.g.\ }

\newcommand{\st}{\mbox{\ s.t.\ }}

\newcommand{\avgFun}{\bar{f}}
\newcommand{\diamG}{\Delta}
\newcommand{\Erdos}{Erd\"os-R\'enyi }

\title{Optimal algorithms for smooth and strongly convex distributed optimization in networks}

\author{
Kevin Scaman$^1$\quad Francis Bach$^2$
\vspace{0.5em}
\\S\'ebastien Bubeck$^3$\quad Yin Tat Lee$^3$\quad Laurent Massouli\'e$^1$
\vspace{1em}
\\$^1$ MSR-INRIA Joint Center, Palaiseau, France
\\$^2$ INRIA, Ecole Normale Sup\'erieure, Paris, France
\\$^3$ Theory group, Microsoft Research, Redmond, United States
\\\normalsize\texttt{\{kevin.scaman,francis.bach,laurent.massoulie\}@inria.fr}
\\\normalsize\texttt{\{yile,sebubeck\}@microsoft.com}
}
\date{}

\begin{document}
\maketitle

\begin{abstract}
In this paper, we determine the optimal convergence rates for strongly convex and smooth distributed optimization in two settings: centralized and decentralized communications over a network. For centralized (\ie \emph{master/slave}) algorithms, we show that distributing Nesterov's accelerated gradient descent is optimal and achieves a precision $\varepsilon>0$ in time $O(\sqrt{\kappa_g}(1+\diamG\tau)\ln(1/\varepsilon))$, where $\kappa_g$ is the condition number of the (global) function to optimize, $\diamG$ is the diameter of the network, and $\tau$ (resp.\ $1$) is the time needed to communicate values between two neighbors (resp.\ perform local computations). For decentralized algorithms based on gossip, we provide the first optimal algorithm, called the \emph{multi-step dual accelerated} (MSDA) method, that achieves a precision $\varepsilon>0$ in time $O(\sqrt{\kappa_l}(1+\frac{\tau}{\sqrt{\gamma}})\ln(1/\varepsilon))$, where $\kappa_l$ is the condition number of the local functions and $\gamma$ is the (normalized) eigengap of the gossip matrix   used for communication between nodes. We then verify the efficiency of MSDA against state-of-the-art methods for two problems: least-squares regression and classification by logistic regression.
\end{abstract} 

\section{Introduction}

Given the numerous applications of distributed optimization in machine learning, many algorithms have recently emerged, that allow the minimization of objective functions $f$ defined as the average $\frac{1}{n} \sum_{i=1}^n f_i$
 of functions $f_i$ which are respectively accessible by separate nodes in a network \cite{nedic2009distributed,boyd2011distributed,duchi2012dual,doi:10.1137/14096668X}. These algorithms typically alternate local incremental improvement steps (such as gradient steps)  with communication steps between nodes in the network, and come with a variety of convergence rates (see for example \cite{shi2014linear,doi:10.1137/14096668X,jakovetic2015linear,2016arXiv160703218N}). 
 
 Two main regimes have been looked at: (a) \emph{centralized} where communications are precisely scheduled and (b) \emph{decentralized} where communications may not exhibit a precise schedule. In this paper, we consider these two regimes for objective functions which are smooth and strongly-convex and for which algorithms are linearly (exponentially) convergent. The main contribution of this paper is to propose new and matching upper and lower bounds of complexity for this class of distributed  problems.
 
 The optimal complexity bounds depend on natural quantities in optimization and network theory. Indeed, (a) for a single machine the optimal number of gradient steps to optimize a function is proportional to the square root of the condition number \cite{nesterov2004introductory}, and (b) for mean estimation, the optimal number of communication steps is proportional to the diameter of the network in centralized problems or to the square root of the eigengap of the Laplacian matrix in decentralized problems \cite{boyd2006randomized}. As shown in \Sec{lb}, our lower complexity bounds happen to be combinations of the two contributions above.

These lower complexity bounds are attained by two separate algorithms. In the centralized case, the trivial distribution of Nesterov's accelerated gradient attains this rate, while in the decentralized case, as shown in \Sec{dual_alg}, the rate is achieved by a dual algorithm. We compare favorably our new optimal algorithms to existing work in \Sec{exps}.

\textbf{Related work.}
Decentralized optimization has been extensively studied and early methods  such as decentralized gradient descent \cite{nedic2009distributed,jakovetic2014fast} or decentralized dual averaging \cite{duchi2012dual} exhibited sublinear convergence rates. More recently, a number of methods with provable linear convergence rates were developed, including EXTRA \cite{doi:10.1137/14096668X,Mokhtari:2016:DDD:2946645.2946706}, augmented Lagrangians \cite{jakovetic2015linear}, and more recent approaches \cite{2016arXiv160703218N}. The most popular of such approaches is the distributed alternating direction method of multipliers (D-ADMM) \cite{boyd2011distributed,wei2012distributed,shi2014linear} and has led to a large number of variations and extensions. In a different direction, second order methods were also investigated \cite{7576649,2016arXiv160606593T}. However, to the best of our knowledge, the field still lacks a coherent theoretical understanding of the optimal convergence rates and its dependency on the characteristics of the communication network.
In several related fields, complexity lower bounds were recently investigated, including the sequential optimization of a sum of functions \cite{DBLP:conf/icml/ArjevaniS16,DBLP:conf/nips/ArjevaniS16}, distributed optimization in flat (i.e. totally connected) networks \cite{Shamir:2014:FLO:2968826.2968845,DBLP:conf/nips/ArjevaniS15}, or distributed stochastic optimization \cite{shamir2014distributed}.

\section{Distributed optimization setting}\label{sec:setting}

\subsection{Optimization problem}
Let $\mathcal{G} = (\mathcal{V},\mathcal{E})$ be a connected simple (\ie undirected) graph of $n$ computing units and diameter $\diamG$, each having access to a function $f_i(\theta)$ over $\theta \in \R^d$. We consider minimizing the average of the local functions
\BEQ\label{eq:global_opt}
\min_{\theta\in\R^d} \avgFun(\theta) = \frac{1}{n} \sum_{i=1}^n f_i(\theta)
\EEQ
in a distributed setting. More specifically, we assume that:
\BNUM
\item Each computing unit can compute first-order characteristics, such as the gradient of its own function or its Fenchel conjugate. By renormalization of the time axis, and without loss of generality, we assume that this computation is performed in one unit of time.
\item Each computing unit can communicate values (\ie vectors in $\R^d$) to its neighbors. This communication requires a time $\tau$ (which may be smaller or greater than~$1$).
\ENUM

These actions may be performed asynchronously and in parallel, and each node $i$ possesses a local version of the parameter, which we refer to as $\theta_i$.
Moreover, we assume that each function $f_i$ is $\alpha$-strongly convex and $\beta$-smooth, and we denote by $\kappa_l = \frac{\beta}{\alpha} \geq 1$ the local condition number. We also denote by $\alpha_g$, $\beta_g$ and $\kappa_g$, respectively, the strong convexity, smoothness and condition number of the average (global) function $\avgFun$. Note that we always have $\kappa_g \leq \kappa_l$, while the opposite inequality is, in general, not true (take for example $f_1(\theta) = \one\{\theta<0\}\theta^2$ and $f_2(\theta) = \one\{\theta>0\}\theta^2$ for which $\kappa_l=+\infty$ and $\kappa_g=1$). However, the two quantities are close (resp.~equal) when the local functions are similar (resp.~equal) to one another.

\subsection{Decentralized communication}

A large body of literature considers a decentralized approach to distributed optimization based on the \emph{gossip} algorithm \cite{boyd2006randomized,nedic2009distributed,duchi2012dual,wei2012distributed}. In such a case, communication is represented as a matrix multiplication with a matrix $W$ verifying the following constraints:
\BNUM
\item $W$ is an $n\times n$ symmetric matrix,
\item $W$ is positive semi-definite,
\item The kernel of $W$ is the set of constant vectors: ${\rm Ker}(W) = {\rm Span}(\one)$, where $\one = (1,...,1)^\top$,
\item $W$ is defined on the edges of the network: $W_{ij} \ne 0$ only if $i=j$ or $(i, j) \in \mathcal{E}$.
\ENUM

The third condition will ensure that the gossip step converges to the average of all the vectors shared between the nodes. We will denote the matrix $W$ as the \emph{gossip matrix}, since each communication step will be represented using it. Note that a simple choice for the gossip matrix is the Laplacian matrix $L = D - A$, where $A$ is the adjacency matrix of the network and $D = \diag \big( {\sum_i A_{ij}} \big)$. However, in the presence of large degree nodes, weighted Laplacian matrices are usually a better choice, and the problem of optimizing these weights is known as the fastest distributed consensus averaging problem and is investigated by \cite{Xiao200465,doi:10.1137/070689413}.

We will denote by $\lambda_1(W)\geq \cdots \geq\lambda_n(W)=0$ the spectrum of the gossip matrix $W$, and its (normalized) \emph{eigengap} the ratio $\gamma(W) = \lambda_{n-1}(W)/\lambda_1(W)$ between the second smallest and the largest eigenvalue. Equivalently, this is the inverse of the condition number of $W$ projected on the space orthogonal to the constant vector $\one$. This quantity will be the main parameter describing the connectivity of the communication network in \Sec{lb_gossip} and \Sec{dual_alg}.

\section{Optimal convergence rates}\label{sec:lb}
In this section, we prove oracle complexity lower bounds for distributed optimization in two settings: strongly convex and smooth functions for centralized (\ie master/slave) and decentralized algorithms based on a gossip matrix $W$.

In the first setting, we show that distributing accelerated gradient descent matches the optimal convergence rate, while, in the second setting, the algorithm proposed in \Sec{dual_alg} is shown to be optimal. Note that we will use the notation $g(\varepsilon)=\Omega(f(\varepsilon))$ for $\exists C>0 \st \forall \varepsilon>0, g(\varepsilon)\geq Cf(\varepsilon)$, and will, for simplicity, omit the additive terms that do not depend on the precision $\varepsilon$ in \Corollary{lb_diam} and \Corollary{gossipOptimal}.

\subsection{Black-box optimization procedures}
The lower bounds provided hereafter depend on a new notion of black-box optimization procedures for the problem in \Eq{global_opt}, where we consider  distributed algorithms verifying the following constraints:
\BNUM
\item \textbf{Local memory:} each node $i$ can store past values in a (finite) internal memory $\mathcal{M}_{i,t}\subset\R^d$ at time $t\geq 0$. These values can be accessed and used at time $t$ by the algorithm run by node $i$, and are updated either by local computation or by communication (defined below), that is, for all $i\in\{1,...,n\}$,
\begin{equation}
\mathcal{M}_{i,t}\subset \mathcal{M}^{comp}_{i,t} \cup \mathcal{M}^{comm}_{i,t}.
\end{equation}

\item \textbf{Local computation:} each node $i$ can, at time $t$, compute the gradient of its local function $\nabla f_i(\theta)$ or its Fenchel conjugate $\nabla f^*_i(\theta)$ for a value $\theta\in\mathcal{M}_{i,t}$ in the node's internal memory, that is, for all $i\in\{1,...,n\}$,
\begin{equation}
\!\!\!
\mathcal{M}^{comp}_{i,t} = {\rm Span}\left(\{\theta,\nabla f_i(\theta),\nabla f^*_i(\theta) : \theta\in\mathcal{M}_{i,t-1}\}\right).
\end{equation}

\item \textbf{Local communication:} each node $i$ can, at time $t$, share a value to all or part of its neighbors, that is, for all $i\in\{1,...,n\}$,

\begin{equation}
\mathcal{M}^{comm}_{i,t} = {\rm Span}\bigg(\bigcup_{(i,j)\in\mathcal{E}}\mathcal{M}_{j,t-\tau}\bigg).
\end{equation}

\item \textbf{Output value:} each node $i$ must, at time $t$, specify one vector in its memory as local output of the algorithm, that is, for all $i\in\{1,...,n\}$,
\begin{equation}
\theta_{i,t} \in \mathcal{M}_{i,t}.
\end{equation}
\ENUM

Hence, a black-box procedure will return $n$ output values---one for each node of the network---and our analysis will focus on ensuring that \emph{all local output values} are converging to the optimal parameter of \Eq{global_opt}.
Moreover, we will say that a black-box procedure \emph{uses a gossip matrix $W$} if the local communication is achieved by multiplication of a vector with $W$. For simplicity, we assume that all nodes start with the simple internal memory $\mathcal{M}_{i,0} = \{0\}$. Note that communications and local computations may be performed in parallel and asynchronously.

\subsection{Centralized algorithms}\label{sec:lb_centralized}

In this section, we show that, for any black-box optimization procedure, at least $\Omega(\sqrt{\kappa_g}\ln(1/\varepsilon))$ gradient steps and $\Omega(\diamG\sqrt{\kappa_g}\ln(1/\varepsilon))$ communication steps are necessary to achieve a precision $\varepsilon>0$, where $\kappa_g$ is the global condition number and $\diamG$ is the diameter of the network. These lower bounds extend the communication complexity lower bounds for totally connected communication networks of \cite{DBLP:conf/nips/ArjevaniS15}, and are natural since at least $\Omega(\sqrt{\kappa_g}\ln(1/\varepsilon))$ steps are necessary to solve a strongly convex and smooth problem up to a fixed precision, and at least $\diamG$ communication steps are required to transmit a message between any given pair of nodes.

In order to simplify the proofs of the following theorems, and following the approach of \cite{bubeck2015convex}, we will consider the limiting situation $d\rightarrow +\infty$. More specifically, we now assume that we are working in $\ell_2 = \{\theta=(\theta_k)_{k\in\N}~:~\sum_k \theta_k^2 < +\infty\}$ rather than $\R^d$.

\begin{theorem}\label{th:global_lb}
Let $\mathcal{G}$ be a graph of diameter $\diamG > 0$ and size $n > 0$, and $\beta_g\geq\alpha_g>0$. There exists $n$ functions $f_i:\ell_2 \rightarrow \R$ such that $\avgFun$ is $\alpha_g$ strongly convex and $\beta_g$ smooth, and for any $t\geq 0$ and any black-box procedure one has, for all $i\in\{1,...,n\}$,

\BEQ
\avgFun(\theta_{i,t}) - \avgFun(\theta^*) \geq \frac{\alpha_g}{2}\left(1 - \frac{4}{\sqrt{\kappa_g}}\right)^{1+\frac{t}{1+\diamG\tau}} \!\!\!\|\theta_{i,0} - \theta^*\|^2,
\EEQ

where $\kappa_g = \beta_g / \alpha_g$.
\end{theorem}
The proof of \Theorem{global_lb} relies on splitting the function used by Nesterov to prove oracle complexities for strongly convex and smooth optimization \cite{nesterov2004introductory,bubeck2015convex} on two nodes at distance $\diamG$. One can show that most dimensions of the parameters $\theta_{i,t}$ will remain zero, and local gradient computations may only increase the number of non-zero dimensions by one. Finally, at least $\diamG$ communication rounds are necessary in-between every gradient computation, in order to share information between the two nodes. The detailed proof is available as supplementary material.
\begin{corollary}\label{cor:lb_diam}
For any graph of diameter $\diamG$ and any black-box procedure, there exists functions $f_i$ such that the time to reach a precision $\varepsilon > 0$ is lower bounded by

\BEQ
\Omega\left(\sqrt{\kappa_g}\Big(1 + \diamG\tau\Big)\ln\left(\frac{1}{\varepsilon}\right)\right),
\EEQ
\end{corollary}

This optimal convergence rate is achieved by distributing Nesterov's accelerated gradient descent on the global function. Computing the gradient of $\avgFun$ is performed by sending all the local gradients $\nabla f_i$ to a single node (denoted as \emph{master node}) in $\diamG$ communication steps (which may involve several simultaneous messages), and then returning the new parameter $\theta_{t+1}$ to every node in the network (which requires another $\diamG$ communication steps). In practice, summing the gradients can be distributed by computing a spanning tree (with the root as master node), and asking for each node to perform the sum of its children's gradients before sending it to its parent. Standard methods as described by \cite{opac-b1080909} can be used for performing this parallelization of gradient computations.

This algorithm has three   limitations: first, the algorithm is not robust to machine failures, and the central role played by the master node also means that a failure of this particular machine may completely freeze the procedure.
Second, and more generally, the algorithm requires precomputing a spanning tree, and is thus not suited to time-varying graphs, in which the connectivity between the nodes may change through time (\eg in \emph{peer-to-peer networks}).
Finally, the algorithm requires \emph{every} node to complete its gradient computation before aggregating them on the master node, and the efficiency of the algorithm thus depends on the slowest of all machines. Hence, in the presence of non-uniform latency of the local computations, or the slow down of a specific machine due to a hardware failure, the algorithm will suffer a significant drop in performance.

\subsection{Decentralized algorithms}\label{sec:lb_gossip}
The gossip algorithm \cite{boyd2006randomized} is a standard method for averaging values across a network when its connectivity may vary through time. This approach was shown to be robust against machine failures, non-uniform latencies and asynchronous or time-varying graphs, and a large body of literature extended this algorithm to distributed optimization \cite{nedic2009distributed,duchi2012dual,wei2012distributed,doi:10.1137/14096668X,jakovetic2015linear,2016arXiv160703218N,7576649}.

The convergence analysis of decentralized algorithms usually relies on the spectrum of the \emph{gossip matrix} $W$ used for communicating values in the network, and more specifically on the ratio between the second smallest and the largest eigenvalue of $W$, denoted $\gamma$.
In this section, we show that, with respect to this quantity and $\kappa_l$, reaching a precision 
$\varepsilon$ requires at least $\Omega(\sqrt{\kappa_l}\ln(1/\varepsilon))$ gradient steps and $\Omega\left(\sqrt{\frac{\kappa_l}{\gamma}}\ln(1/\varepsilon)\right)$ communication steps, by exhibiting a gossip matrix such that a corresponding lower bound exists.

\begin{theorem}\label{th:local_lb}
Let $\alpha,\beta > 0$ and $\gamma\in(0,1]$. There exists a gossip matrix $W$ of eigengap $\gamma(W)=\gamma$, and $\alpha$-strongly convex and $\beta$-smooth functions $f_i:\ell_2 \rightarrow \R$ such that, for any $t\geq 0$ and any black-box procedure using $W$ one has, for all $i\in\{1,...,n\}$,
\BEQ
\avgFun(\theta_{i,t}) - \avgFun(\theta^*) \geq \frac{3\alpha}{2}\left(1 - \frac{16}{\sqrt{\kappa_l}}\right)^{1+\frac{t}{1+\frac{\tau}{5\sqrt{\gamma}}}} \|\theta_{i,0} - \theta^*\|^2,
\EEQ
where $\kappa_l = \beta / \alpha$ is the local condition number.
\end{theorem}
The proof of \Theorem{local_lb} relies on the same technique as that of \Theorem{global_lb}, except that we now split the two functions on a subset of a linear graph. These networks have the appreciable property that $\diamG\approx 1/\sqrt{\gamma}$, and we can thus use a slightly extended version of \Theorem{global_lb} to derive the desired result. The complete proof is available as supplementary material.

\begin{corollary}\label{cor:gossipOptimal}
For any $\gamma > 0$, there exists a gossip matrix $W$ of eigengap $\gamma$ and $\alpha$-strongly convex, $\beta$-smooth functions such that, with $\kappa_l=\beta/\alpha$, for any black-box procedure using~$W$ the time to reach a precision $\varepsilon > 0$ is lower bounded by

\BEQ
\Omega\left(\sqrt{\kappa_l}\left(1 + \frac{\tau}{\sqrt{\gamma}}\right)\ln\left(\frac{1}{\varepsilon}\right)\right).
\EEQ
\end{corollary}
We will see in the next section that this lower bound is met for a novel decentralized algorithm called \emph{multi-step dual accelerated} (MSDA) and based on the dual formulation of the optimization problem.
Note that these results provide optimal convergence rates with respect to $\kappa_l$ and $\gamma$, but do not imply that $\gamma$ is the right quantity to consider on general graphs. The quantity $1/\sqrt{\gamma}$ may indeed be very large compared to $\diamG$, for example for star networks, for which $\diamG = 2$ and $1/\sqrt{\gamma}=\sqrt{n}$. However, on many simple networks, the diameter $\diamG$ and the eigengap of the Laplacian matrix are tightly connected, and $\diamG\approx 1/\sqrt{\gamma}$. For example, for linear graphs, $\diamG = n-1$ and $1/\sqrt{\gamma}\approx 2n/\pi$, for totally connected networks, $\diamG = 1$ and $1/\sqrt{\gamma}=1$, and for regular networks, $1/\sqrt{\gamma}\geq \frac{\diamG}{2\sqrt{2}\ln_2{n}}$ \cite{isoperimetric}.
Finally, note that the case of totally connected networks corresponds to a previous complexity lower bound on communications proven by \cite{DBLP:conf/nips/ArjevaniS15}, and is equivalent to our result for centralized algorithms with $\diamG=1$.

\section{Optimal decentralized algorithms}\label{sec:dual_alg}
In this section, we present a simple framework for solving the optimization problem in \Eq{global_opt} in a decentralized setting, from which we will derive several variants, including a  synchronized algorithm whose convergence rate matches the lower bound in \Corollary{gossipOptimal}
. Note that the naive approach of distributing each (accelerated) gradient step by gossiping does not lead to a linear convergence rate, as the number of gossip steps has to increase with the number of iterations to ensure the linear rate is preserved. We begin with the simplest form of the algorithm, before extending it to more advanced scenarios.

\subsection{Single-Step Dual Accelerated method}

A standard approach for solving \Eq{global_opt} (see \cite{boyd2011distributed,jakovetic2015linear}) consists in rewriting the optimization problem as
\BEQ
\min_{\theta\in\R^d} \avgFun(\theta) = \min_{\theta_1=\cdots=\theta_n} \frac{1}{n} \sum_{i=1}^n f_i(\theta_i).
\EEQ

Furthermore, the equality constraint $\theta_1=\cdots=\theta_n$ is equivalent to $\Theta\sqrt{W}=0$, where $\Theta=(\theta_1,\ldots,\theta_n)$ and $W$ is a gossip matrix verifying the assumptions described in \Sec{setting}. Note that, since $W$ is positive semi-definite, $\sqrt{W}$ exists and is defined as $\sqrt{W} = V^\top\Sigma^{1/2} V$, where $W=V^\top\Sigma V$ is the singular value decomposition of $W$. The equality $\Theta\sqrt{W} = 0$ implies that each row of $\Theta$ is constant (since ${\rm Ker}(\sqrt{W}) = {\rm Span}(\one)$), and is thus equivalent to $\theta_1=\cdots=\theta_n$.
This leads to the following \emph{primal version} of the optimization problem:
\BEQ\label{eq:primal_opt}
\min_{\Theta\in\R^{d\times n}~:~\Theta\sqrt{W} = 0} F(\Theta),
\EEQ
where $F(\Theta) = \sum_{i=1}^n f_i(\theta_i)$. Since \Eq{primal_opt} is a convex problem, it is equivalent to its \emph{dual optimization problem}:
\BEQ\label{eq:dual_opt}
\max_{\lambda\in\R^{d\times n}} -F^*(\lambda\sqrt{W}),
\EEQ
where $F^*(y) = \sup_{x\in\R^{d\times n}} \scprod{y}{x} - F(x)$ is the Fenchel conjugate of $F$, and $\scprod{y}{x} = {\rm tr}(y^\top x)$ is the standard scalar product between matrices.

The optimization problem in \Eq{dual_opt} is unconstrained and convex, and can thus be solved using a variety of convex optimization techniques. The proposed \emph{single-step dual accelerated} (SSDA) algorithm described in \Alg{acc_dual} uses Nesterov's accelerated gradient descent, and can be thought of as an accelerated version of the distributed augmented Lagrangian method of \cite{jakovetic2015linear} for $\rho=0$. The algorithm is derived by noting that a gradient step of size $\eta > 0$ for \Eq{dual_opt} is
\BEQ
\lambda_{t+1} = \lambda_t - \eta \nabla F^*(\lambda_t \sqrt{W})\sqrt{W},
\EEQ
and the change of variable $y_t = \lambda_t \sqrt{W}$ leads to
\BEQ
y_{t+1} = y_t - \eta \nabla F^*(y_t)W.
\EEQ
This equation can be interpreted as gossiping the gradients of the local conjugate functions $\nabla f_i^*(y_{i,t})$, since $\nabla F^*(y_t)_{ij} = \nabla f_j^*(y_{j,t})_i$.

\begin{algorithm}[t]
\caption{Single-Step Dual Accelerated method}
\label{alg:acc_dual}
\begin{algorithmic}[1]
	\REQUIRE{number of iterations $T>0$, gossip matrix $W\in\R^{n\times n}$, $\eta = \frac{\alpha}{\lambda_1(W)}$, $\mu = \frac{\sqrt{\kappa_l} - \sqrt{\gamma}}{\sqrt{\kappa_l} + \sqrt{\gamma}}$}
  \ENSURE{$\theta_{i,T}$, for $i=1,...,n$}
  \STATE $x_0 = 0$, $y_0 = 0$
  \FOR{$t = 0$ to $T-1$}
		\STATE $\theta_{i,t} = \nabla f_i^*(x_{i,t})$, for all $i = 1,...,n$
		\STATE $y_{t+1} = x_t - \eta \Theta_t W$
		\STATE $x_{t+1} = (1+\mu)y_{t+1} - \mu y_t$
  \ENDFOR
\end{algorithmic}
\end{algorithm}

\begin{theorem}\label{th:cv_SSDA}
The iterative scheme in \Alg{acc_dual} converges to $\Theta = {\theta^*}\one^\top$ where $\theta^*$ is the solution of \Eq{global_opt}. Furthermore, the time needed for this algorithm to reach any given precision $\varepsilon > 0$ is
\BEQ
O\left((1+\tau)\sqrt{\frac{\kappa_l}{\gamma}}\ln\left(\frac{1}{\varepsilon}\right)\right).
\EEQ
\end{theorem}
This theorem relies on proving that the condition number of the dual objective function is upper bounded by $\frac{\kappa_l}{\gamma}$, and noting that the convergence rate for accelerated gradient descent depends on the square root of the condition number (see, e.g., \cite{bubeck2015convex}). A detailed proof is available as supplementary material.

\subsection{Multi-Step Dual Accelerated method}
The main problem of \Alg{acc_dual} is that it always performs the same number of gradient and gossip steps. When communication is cheap compared to local computations ($\tau \ll 1$), it would be preferable to perform more gossip steps than gradient steps in order to propagate the local gradients further than the local neighborhoods of each node. This can be achieved by replacing $W$ by $P_K(W)$ in \Alg{acc_dual}, where $P_K$ is a polynomial of degree at most $K$. 
If $P_K(W)$ is itself a gossip matrix, then the analysis of the previous section can be applied and the convergence rate of the resulting algorithm depends on the eigengap of $P_K(W)$.
Maximizing this quantity for a fixed $K$ leads to a common acceleration scheme known as  Chebyshev acceleration \cite{chebyshevAcc,Arioli:2014:CAI:2638909.2639021} and the choice
\BEQ
P_K(x) = 1 - \frac{T_K(c_2(1-x))}{T_K(c_2)},
\EEQ
where $c_2=\frac{1+\gamma}{1-\gamma}$ and $T_K$ are the Chebyshev polynomials \cite{chebyshevAcc} defined as $T_0(x) = 1$, $T_1(x) = x$, and, for all $k\geq 1$,
\BEQ
T_{k+1}(x) = 2xT_k(x)-T_{k-1}(x).
\EEQ
Finally, verifying that this particular choice of $P_K(W)$ is indeed a gossip matrix, and taking $K=\lfloor\frac{1}{\sqrt{\gamma}}\rfloor$ leads to \Alg{multi_acc_dual} with an optimal convergence rate with respect to $\gamma$ and $\kappa_l$.
\begin{theorem}\label{th:MSDA_cv}
The iterative scheme in \Alg{multi_acc_dual} converges to $\Theta = {\theta^*}\one^\top$ where $\theta^*$ is the solution of \Eq{global_opt}. Furthermore, the time needed for this algorithm to reach any given precision $\varepsilon > 0$ is
\BEQ
O\left(\sqrt{\kappa_l}\left(1+\frac{\tau}{\sqrt{\gamma}}\right)\ln\left(\frac{1}{\varepsilon}\right)\right).
\EEQ
\end{theorem}
The proof of \Theorem{MSDA_cv} relies on standard properties of Chebyshev polynomials that imply that, for the particular choice of $K=\lfloor\frac{1}{\sqrt{\gamma}}\rfloor$, we have $\frac{1}{\sqrt{\gamma(P_K(W))}} \leq 2$. Hence, \Theorem{cv_SSDA} applied to the gossip matrix $W' = P_K(W)$ gives the desired convergence rate. The complete proof is available as supplementary material.

\begin{algorithm}[t]
\caption{Multi-Step Dual Accelerated method}
\label{alg:multi_acc_dual}
\begin{algorithmic}[1]
	\REQUIRE{number of iterations $T>0$, gossip matrix $W\in\R^{n\times n}$, $c_1=\frac{1-\sqrt{\gamma}}{1+\sqrt{\gamma}}$, $c_2=\frac{1+\gamma}{1-\gamma}$, $c_3=\frac{2}{(1+\gamma)\lambda_1(W)}$, $K = \left\lfloor\frac{1}{\sqrt{\gamma}}\right\rfloor$, $\eta = \frac{\alpha(1+c_1^{2K})}{(1+c_1^K)^2}$, $\mu = \frac{(1+c_1^K)\sqrt{\kappa_l} - 1 + c_1^K}{(1+c_1^K)\sqrt{\kappa_l} + 1 - c_1^K}$}
  \ENSURE{$\theta_{i,T}$, for $i=1,...,n$}
  \STATE $x_0 = 0$, $y_0 = 0$
  \FOR{$t = 0$ to $T-1$}
		\STATE $\theta_{i,t} = \nabla f_i^*(x_{i,t})$, for all $i = 1,...,n$
		\STATE $y_{t+1} = x_t - \eta$ \textsc{acceleratedGossip}($\Theta_t$,$W$,$K$)
		\STATE $x_{t+1} = (1+\mu)y_{t+1} - \mu y_t$
  \ENDFOR
	\vspace{0.5em}
	\STATE \textbf{procedure} \textsc{acceleratedGossip}($x$,$W$,$K$)
	\STATE $a_0 = 1$,	$a_1 = c_2$
	\STATE $x_0 = x$, $x_1 = c_2x(I - c_3W)$
	\FOR{$k = 1$ to $K-1$}
		\STATE $a_{k+1} = 2c_2a_k - a_{k-1}$
		\STATE $x_{k+1} = 2c_2x_k(I - c_3W) - x_{k-1}$
  \ENDFOR
	\STATE \textbf{return} $x_0 - \frac{x_K}{a_K}$
	\STATE \textbf{end procedure}
\end{algorithmic}
\end{algorithm}

\subsection{Discussion and further developments}
We now discuss several extensions to the proposed algorithms.

\BNUM
\item \textbf{Computation of $\nabla f_i^*(x_{i,t})$:} In practice, it may be hard to apply the dual algorithm when  conjugate functions are hard to compute. We now provide three potential solutions to this problem: (1) \emph{warm starts} may be used for the optimization problem $\nabla f_i^*(x_{i,t}) = \argmin_\theta f_i(\theta) - x_{i,t}^\top\theta$ by starting from the previous iteration $\theta_{i,t-1}$. This will drastically reduce the number of steps required for convergence. (2) SSDA and MSDA can be extended to composite functions of the form $f_i(\theta) = g_i(B_i \theta) + c \|\theta\|_2^2 $ for $B_i \in \R^{m_i \times d}$ and $g_i$ smooth, and for which we know how to compute the proximal operator. This allows applications in machine learning such as logistic regression. See supplementary material for details. (3) 
Beyond the composite case, one can also add a small (well-chosen) quadratic term to the dual, and by applying accelerated gradient descent on the corresponding primal, get an algorithm that uses primal gradient computations and achieves almost the same guarantee as SSDA and MSDA (off by a $\log(\kappa_l/\gamma)$ factor).

\item \textbf{Local vs. global condition number:} MSDA and SSDA depend on the worst strong convexity of the local functions $\min_i \alpha_i$, which may be very small. A simple trick can be used to depend on the \emph{average} strong convexity. Using the proxy functions $g_i(\theta) = f_i(\theta) - (\alpha_i - \bar{\alpha})\|\theta\|_2^2$ instead of $f_i$, where $\bar{\alpha} = \frac{1}{n}\sum_i \alpha_i$ is the average strong convexity, will improve the local condition number from $\kappa_l = \frac{\max_i \beta_i}{\min_i \alpha_i}$ to

\BEQ
\kappa_l' = \frac{\max_i \beta_i - \alpha_i}{\bar{\alpha}} - 1.
\EEQ

Several algorithms, including EXTRA \cite{doi:10.1137/14096668X} and DIGing \cite{2016arXiv160703218N}, have convergence rates that depend on the strong convexity of the global function $\alpha_g$. However, their convergence rates are not optimal, and it is still an open question to know if a rate close to $O\left(\sqrt{\kappa_g}(1+\frac{\tau}{\sqrt{\gamma}})\ln(1/\varepsilon)\right)$ can be achieved with a decentralized algorithm.

\item \textbf{Asynchronous setting:} Accelerated stochastic gradient descent such as SVRG \cite{NIPS2013_4937} or SAGA \cite{defazio2014saga} can be used on the dual problem in \Eq{dual_opt} instead of accelerated gradient descent, in order to obtain an asynchronous algorithm with a linear convergence rate. The details and exact convergence rate of such an approach are left as future work.
\ENUM

\section{Experiments}\label{sec:exps}
In this section, we compare our new algorithms, single-step dual accelerated (SSDA) descent and multi-step dual accelerated (MSDA) descent, to standard distributed optimization algorithms in two settings: least-squares regression and  classification by logistic regression. Note that these experiments on simple generated datasets are made to assess the differences between existing state-of-the-art algorithms and the ones provided in \Sec{dual_alg}, and do not address the practical implementation details nor the efficiency of the compared algorithms on real-world distributed platforms. The effect of latency, machine failures or variable communication time is thus left for future work.
\subsection{Competitors and setup}
We compare SSDA and MSDA to four state-of-the-art distributed algorithms that achieve linear convergence rates:
distributed ADMM (D-ADMM) \cite{shi2014linear}, EXTRA \cite{doi:10.1137/14096668X}, a recent approach named DIGing \cite{2016arXiv160703218N}, and the distributed version of accelerated gradient descent (DAGD) described in \Sec{lb_centralized} and shown to be optimal among centralized algorithms.
When available in the literature, we used the optimal parameters for each algorithm (see Theorem 2 by \cite{shi2014linear} for D-ADMM and Remark 3 by \cite{doi:10.1137/14096668X} for EXTRA). For the DIGing algorithm, the parameters provided by \cite{2016arXiv160703218N} are very conservative, and lead to a very slow convergence. We thus manually optimized the parameter for this algorithm.
The experiments are simulated using a generated dataset consisting of $10,000$ samples randomly distributed to the nodes of a network of size $100$. In order to assess the effect of the connectivity of the network, we ran each experiment on two networks: one $10\times 10$ grid and an \Erdos random network with parameter $p=\frac{6}{100}$ (\ie of average degree $6$). The quality metric used in this section is be the maximum approximation error among the nodes of the network
\BEQ
e_t = \max_{i\in\mathcal{V}} \avgFun(\theta_{i,t}) - \avgFun(\theta^*),
\EEQ
where $\theta^*$ is the optimal parameter of the optimization problem in \Eq {global_opt}.

\subsection{Least-squares regression}

\begin{figure}[t]
	\centering
	\subfigure[high communication time: {$\tau=10$}]{
		\includegraphics[viewport=0 0 278 196, clip=true, width=0.45\textwidth]{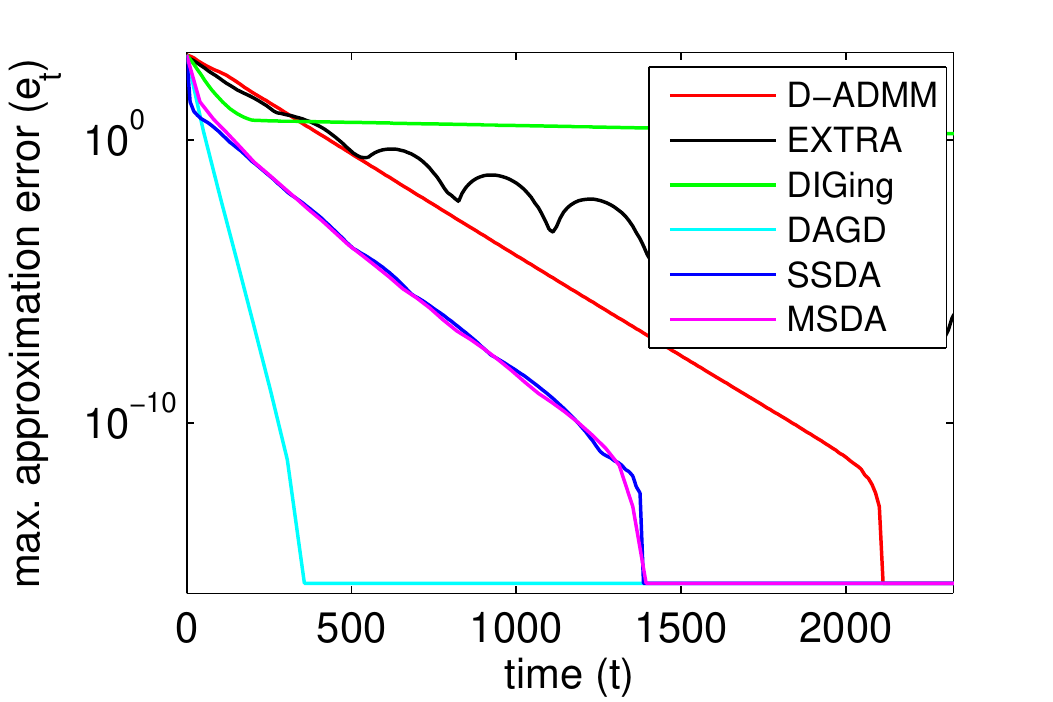}
	}
	\subfigure[low communication time: {$\tau=0.1$}]{
		\includegraphics[viewport=0 0 278 196, clip=true, width=0.45\textwidth]{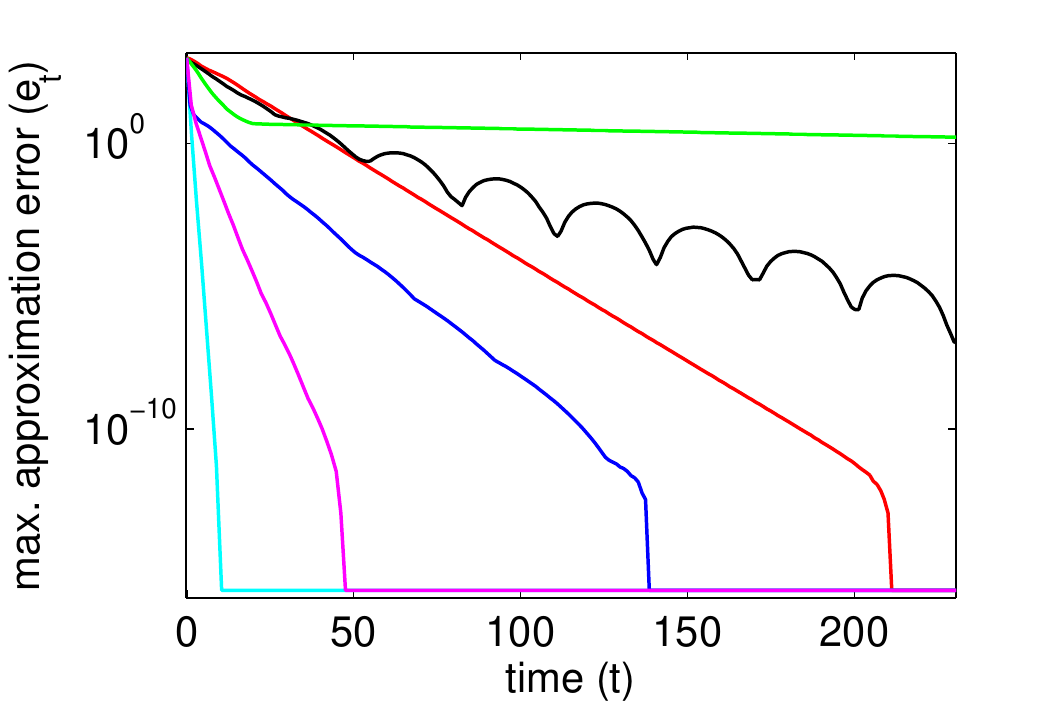}
	}
	\caption{Maximum approximation error for least-squares regression on an \Erdos random network of average degree $6$ ($n=100$).}
	\label{fig:lsr1}
\end{figure}

\begin{figure}[t]
	\centering
	\subfigure[high communication time: {$\tau=10$}]{
		\includegraphics[viewport=0 0 278 196, clip=true, width=0.45\textwidth]{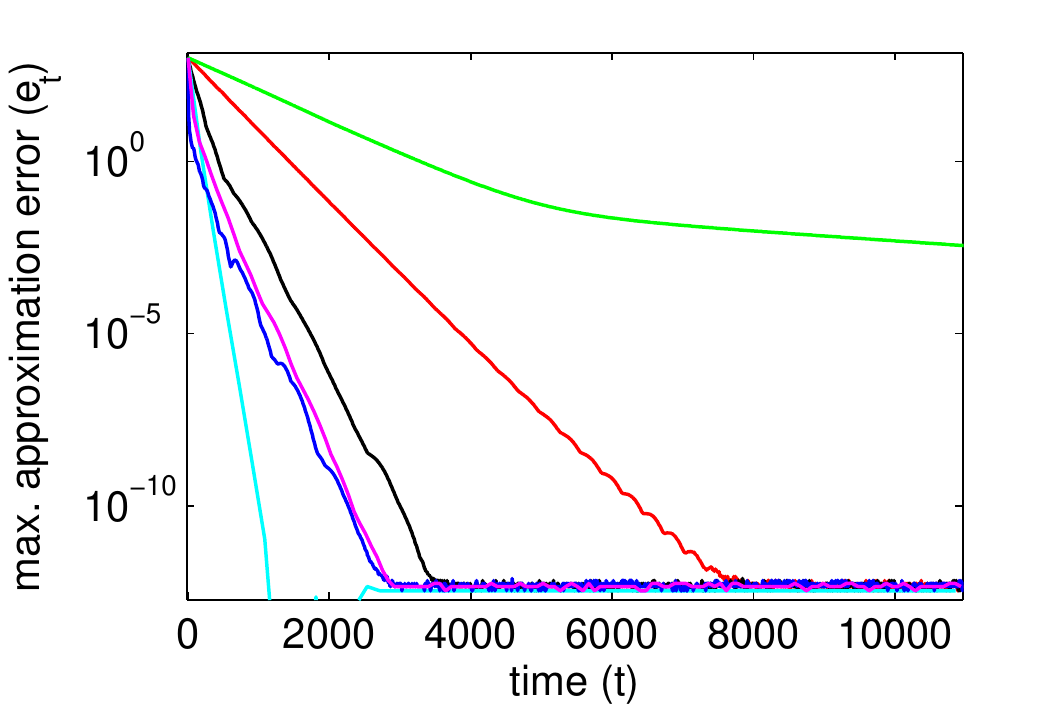}
	}
	\subfigure[low communication time: {$\tau=0.1$}]{
		\includegraphics[viewport=0 0 278 196, clip=true, width=0.45\textwidth]{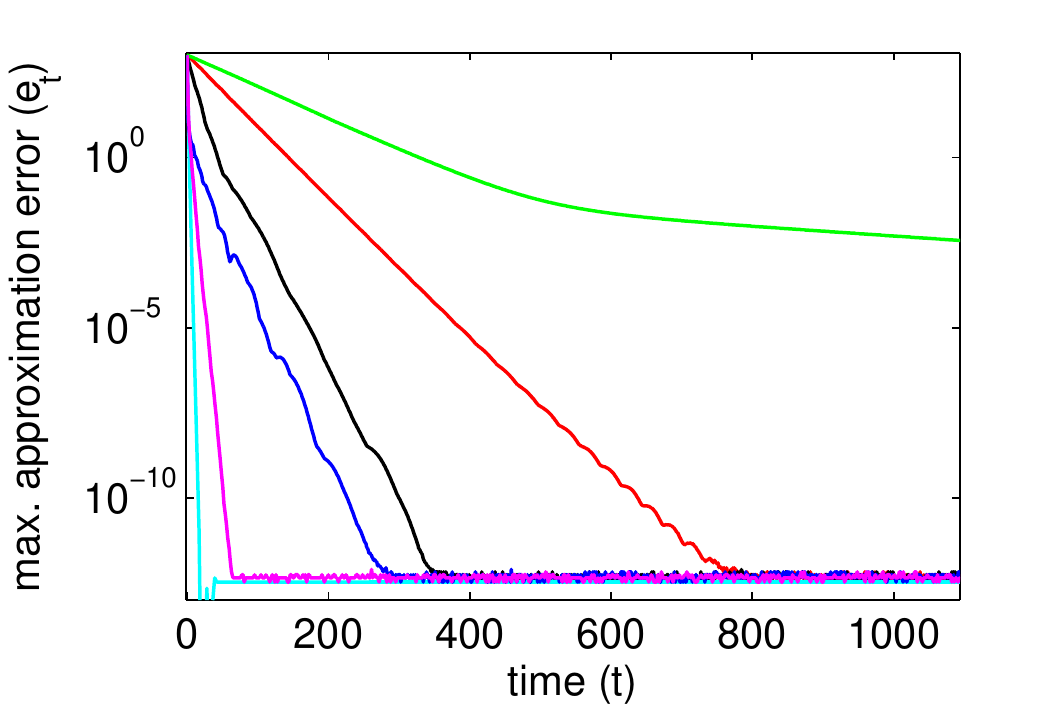}
	}
	\caption{Maximum approximation error for least-squares regression on a $10\times 10$ grid graph ($n=100$).}
	\label{fig:lsr2}
\end{figure}

The \emph{regularized least-squares regression} problem consists in solving the optimization problem
\BEQ
\min_{\theta\in\R^d} \frac{1}{m} \|y - X^\top\theta\|_2^2 + c\|\theta\|_2^2,
\EEQ
where $X\in\R^{d\times m}$ is a matrix containing the $m$ data points, and $y\in\R^m$ is a vector containing the $m$ associated values. The task is thus to minimize the empirical quadratic error between a function $y_i=g(X_i)$ of $d$ variables and its linear regression $\hat{g}(X_i) = X_i^\top\theta$ on the original dataset (for $i=1,...,m$), while smoothing the resulting approximation by adding a regularizer $c\|\theta\|_2^2$. For our experiments, we fixed $c=0.1$, $d=10$, and sampled $m=10,000$ Gaussian random variables $X_i\sim\mathcal{N}(0,1)$ of mean $0$ and variance $1$. The function to regress is then $y_i=X_i^\top\one + \cos(X_i^\top\one) + \xi_i$ where $\xi_i\sim\mathcal{N}(0,1/4)$ is an i.i.d.~Gaussian noise of variance $1/4$. These data points are then distributed randomly and evenly to the $n=100$ nodes of the network. Note that the choice of function to regress $y$ does not impact the Hessian of the objective function, and thus the convergence rate of the optimization algorithms.

\Fig{lsr1} and \Fig{lsr2} show the performance of the compared algorithms on two networks: a $10\times 10$ grid graph and an \Erdos random graph of average degree $6$. All algorithms are linearly convergent, although their convergence rates scale on several orders of magnitude. In all experiments, the centralized optimal algorithm DAGD has the best convergence rate, while MSDA has the best convergence rate among decentralized methods. When the communication time is smaller than the computation time ($\tau\gg 1$), performing several communication rounds per gradient iteration will improve the efficiency of the algorithm and MSDA substantially outperforms SSDA.

\subsection{Logistic classification}

\begin{figure}[t]
	\centering
	\subfigure[high communication time: {$\tau=10$}]{
		\includegraphics[viewport=0 0 278 196, clip=true, width=0.45\textwidth]{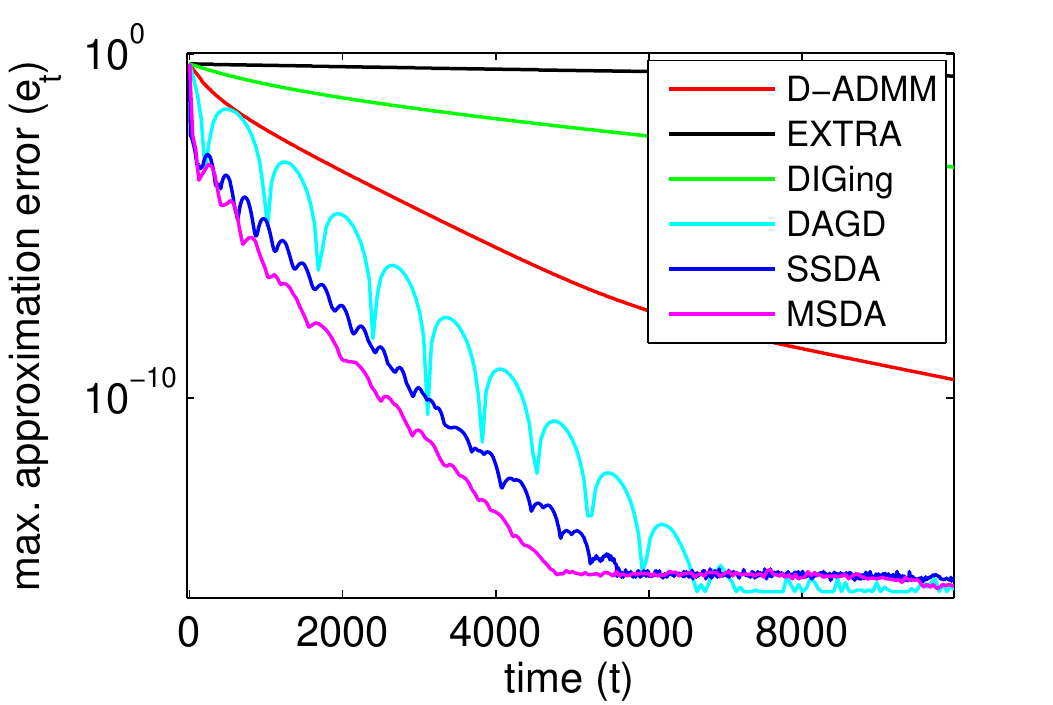}
	}
	\subfigure[low communication time: {$\tau=0.1$}]{
		\includegraphics[viewport=0 0 278 196, clip=true, width=0.45\textwidth]{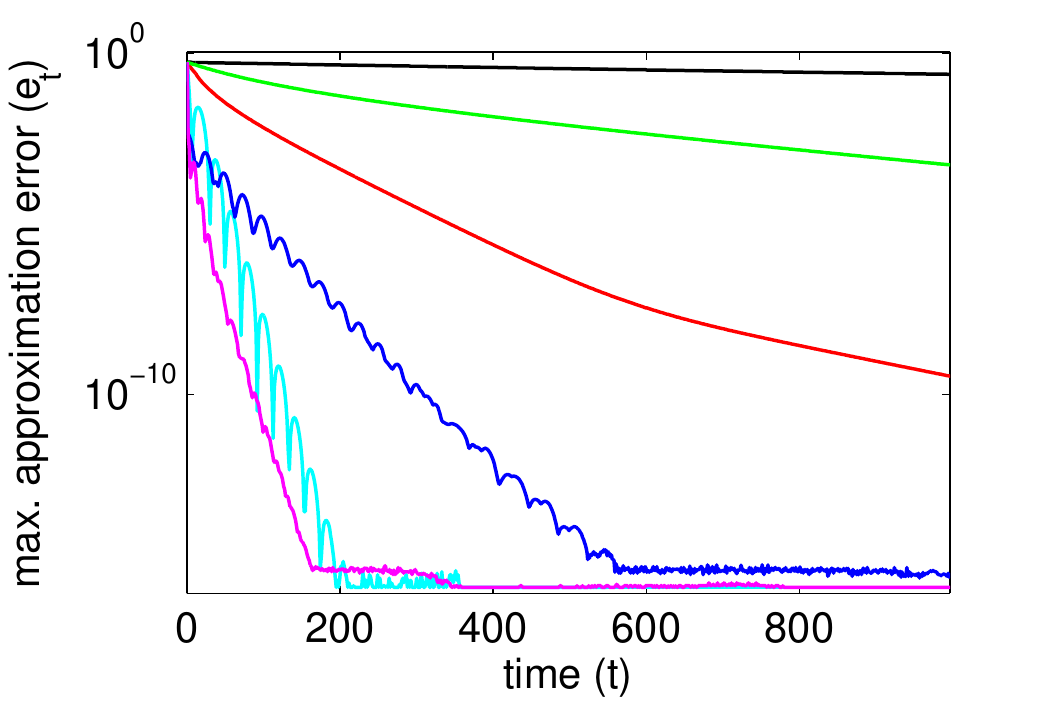}
	}
	\caption{Maximum approximation error for logistic classification on an \Erdos random network of average degree $6$ ($n=100$).}
	\label{fig:lc1}
\end{figure}

\begin{figure}[t]
	\centering
	\subfigure[high communication time: {$\tau=10$}]{
		\includegraphics[viewport=0 0 278 196, clip=true, width=0.45\textwidth]{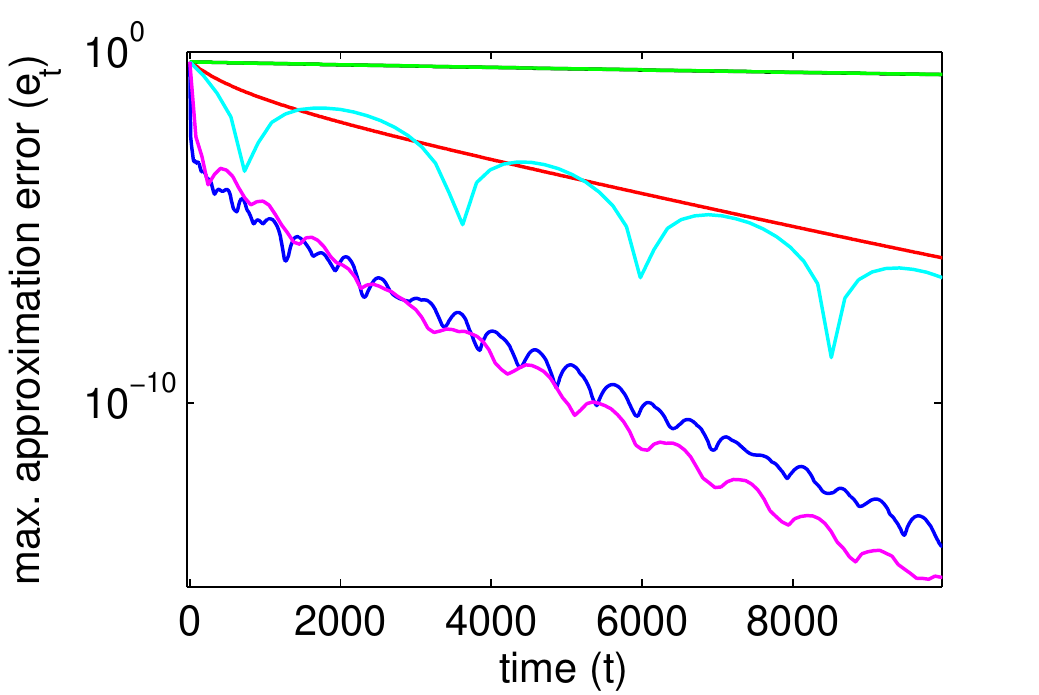}
	\label{fig:lc2a}
	}
	\subfigure[low communication time: {$\tau=0.1$}]{
		\includegraphics[viewport=0 0 278 196, clip=true, width=0.45\textwidth]{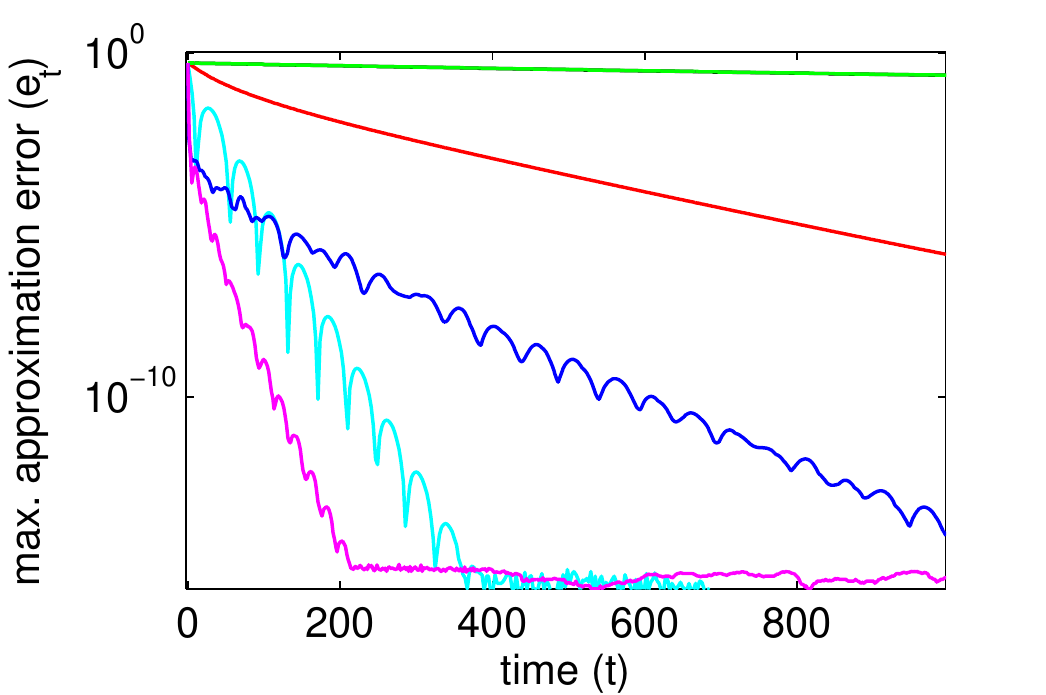}
	}
	\caption{Maximum approximation error for logistic classification on a $10\times 10$ grid graph ($n=100$).}
	\label{fig:lc2}
\end{figure}

The \emph{logistic classification} problem consists in solving the optimization problem
\BEQ
\min_{\theta\in\R^d} \frac{1}{m}\sum_{i=1}^m\ln\left(1 + e^{-y_i \cdot X_i^\top\theta}\right) + c\|\theta\|_2^2,
\EEQ
where $X\in\R^{d\times m}$ is a matrix containing $m$ data points, and $y\in\{-1,1\}^m$ is a vector containing the $m$ class assignments. The task is thus to classify a dataset by learning a linear classifier mapping data points $X_i$ to their associated class $y_i\in\{-1,1\}$. For our experiments, we fixed $c=0.1$, $d=10$, and sampled $m=10,000$ data points, $5,000$ for the first class and $5,000$ for the second. Each data point $X_i\sim\mathcal{N}(y_i\one,1)$ is a Gaussian random variable of mean $y_i\one$ and variance $1$, where $y_i=2\one\{i\leq 5,000\}-1$ is the true class of $X_i$. These data points are then distributed randomly and evenly to the $n=100$ nodes of the network.

\Fig{lc1} and \Fig{lc2} show the performance of the compared algorithms for logistic classification on two networks: a $10\times 10$ grid graph and an \Erdos random graph of average degree $6$. As for least-squares regression, all algorithms are linearly convergent, and their convergence rates scale on several orders of magnitude. In this case, the centralized optimal algorithm DAGD is outperformed by MSDA, although the two convergence rates are relatively similar. Again, when the communication time is smaller than the computation time ($\tau\gg 1$), performing several communication rounds per gradient iteration will improve the efficiency of the algorithm and MSDA substantially outperforms SSDA.
Note that, in \Fig{lc2a}, D-ADMM requires $383$ iterations to reach the same error obtained after only $10$ iterations of SSDA, demonstrating a substantial improvement over state-of-the-art methods.

\section{Conclusion}
In this paper, we derived optimal convergence rates for strongly convex and smooth distributed optimization in two settings: centralized and decentralized communications in a network.
For the decentralized setting, we introduced the \emph{multi-step dual accelerated} (MSDA) algorithm with a provable optimal linear convergence rate, and showed its high efficiency compared to other state-of-the-art methods, including distributed ADMM and EXTRA. The simplicity of the approach makes the algorithm extremely flexible, and allows for future extensions, including time-varying networks and an analysis for non-strongly-convex functions. Finally, extending our complexity lower bounds to time delays, variable computational speeds of local systems, or machine failures would be a notable addition to this work.

\appendix

\section{Detailed proofs}
\subsection{Complexity lower bounds}
\begin{proof}[Proof of Theorem 1]
This proof relies on splitting the function used by Nesterov to prove oracle complexities for strongly convex and smooth optimization \cite{nesterov2004introductory,bubeck2015convex}. Let $\beta\geq\alpha>0$, $\mathcal{G}=(\mathcal{V},\mathcal{E})$ a graph and $A\subset\mathcal{V}$ a set of nodes of $\mathcal{G}$. For all $d>0$, we denote as $A_d^c = \{v\in\mathcal{V}~:~d(A,v)\geq d\}$ the set of nodes at distance at least $d$ from $A$, and let, for all $i\in\mathcal{V}$, $f^A_i:\ell_2\rightarrow\R$ be the functions defined as:
\BEQ
\!\!\!\!\!\!f^A_i(\theta) = \left\{\!\!\!
\BA{ll}
\frac{\alpha}{2n}\|\theta\|_2^2 + \frac{\beta - \alpha}{8|A|}(\theta^\top M_1 \theta - \theta_1) \!\!\!\!&\mbox{if }i\in A\\
\frac{\alpha}{2n}\|\theta\|_2^2 + \frac{\beta - \alpha}{8|A_d^c|} \theta^\top M_2 \theta &\mbox{if }i\in A_d^c\\
\frac{\alpha}{2n}\|\theta\|_2^2 &\mbox{otherwise}\\
\EA\right.
\EEQ
where $M_1:\ell_2\rightarrow\ell_2$ is the infinite block diagonal matrix with $\Big(\BA{cc}1 &-1\\-1 &1\EA\Big)$ on the diagonal, and $M_2 = \Big(\BA{cc}1 &0\\0 &M_1\EA\Big)$.
First, note that, since $0 \preceq M_1+M_2 \preceq 4I$, $\avgFun^A=\frac{1}{n} \sum_{i=1}^n f^A_i$ is $\alpha$-strongly convex and $\beta$-smooth. Then, Theorem 1 is a direct consequence of the following lemma:
\begin{lemma}\label{lem:A_lb}
If $A_d^c\neq\emptyset$, then for any $t\geq 0$ and any black-box procedure one has, for all $i\in\{1,...,n\}$,
\BEQ
\avgFun^A(\theta_{i,t}) - \avgFun^A(\theta^*) \geq \frac{\alpha}{2}\left(\frac{\sqrt{\kappa_g} - 1}{\sqrt{\kappa_g} + 1}\right)^{2\left(1+\frac{t}{1+d\tau}\right)} \|\theta_{i,0} - \theta^*\|^2,
\EEQ
where $\kappa_g = \beta / \alpha$.
\end{lemma}
\begin{proof}
This lemma relies on the fact that most of the coordinates of the vectors in the memory of any node will remain equal to $0$. More precisely, let $k_{i,t} = \max\{k\in\N~:~\exists\theta\in\mathcal{M}_{i,t}\st\theta_k\neq0\}$ be the last non-zero coordinate of a vector in the memory of node $i$ at time $t$. Then, under any black-box procedure, we have, for any local computation step,
\BEQ
k_{i,t+1} \leq \left\{
\BA{ll}
k_{i,t} + \one\{k_{i,t}\equiv 0 \mbox{ mod } 2\} &\mbox{if } i\in A\\
k_{i,t} + \one\{k_{i,t}\equiv 1 \mbox{ mod } 2\} &\mbox{if } i\in A_d^c\\
k_{i,t} &\mbox{otherwise}\\
\EA\right..
\EEQ
Indeed, local gradients can only increase even dimensions for nodes in $A$ and odd dimensions for nodes in $A_d^c$. The same holds for gradients of the dual functions, since these have the same block structure as their convex conjugates. Thus, in order to reach the third coordinate, algorithms must first perform one local computation in $A$, then $d$ communication steps in order for a node in $A_d^c$ to have a non-zero second coordinate, and finally, one local computation in $A_d^c$. Accordingly, one must perform at least $k$ local computation steps and $(k-1)d$ communication steps to achieve $k_{i,t} \geq k$ for at least one node $i\in\mathcal{V}$, and thus, for any $k\in\N^*$,
\BEQ
\forall t < 1 + (k-1)(1+d\tau), k_{i,t}\leq k-1.
\EEQ
This implies in particular:
\BEQ\label{eq:lem1}
\forall i\in\mathcal{V}, k_{i,t} \leq \left\lfloor\frac{t-1}{1+d\tau}\right\rfloor+1\le \frac{t}{1+d\tau}+1.
\EEQ
Furthermore, by definition of $k_{i,t}$, one has $\theta_{i,k} = 0$ for all $k>k_{i,t}$, and thus
\BEQ\label{eq:lem2}
\|\theta_{i,t} - \theta^*\|_2^2 \geq \sum_{k=k_{i,t}+1}^{+\infty} {\theta^*_k}^2.
\EEQ
and, since $\avgFun^A$ is $\alpha$-strongly convex,
\BEQ\label{eq:lem3}
\avgFun^A(\theta_{i,t}) - \avgFun^A(\theta^*) \geq \frac{\alpha}{2}\|\theta_{i,t} - \theta^*\|_2^2.
\EEQ
Finally, the solution of the global problem $\min_\theta \avgFun^A(\theta)$ is $\theta^*_k = \left(\frac{\sqrt{\beta} - \sqrt{\alpha}}{\sqrt{\beta} + \sqrt{\alpha}}\right)^k$. Combining this result with Eqs.~(\ref{eq:lem1}), (\ref{eq:lem2}) and (\ref{eq:lem3}) leads to the desired inequality.
\end{proof}
Using the previous lemma with $d=\diamG$ the diameter of $\mathcal{G}$ and $A=\{v\}$ one of the pair of nodes at distance $\diamG$ returns the desired result.
\end{proof}

\begin{proof}[Proof of Theorem 2]
Let $\gamma_n = \frac{1-\cos(\frac{\pi}{n})}{1+\cos(\frac{\pi}{n})}$ be a decreasing sequence of positive numbers. Since $\gamma_2 = 1$ and $\lim_n \gamma_n = 0$, there exists $n\geq 2$ such that $\gamma_n\geq\gamma > \gamma_{n+1}$. The cases $n=2$ and $n\geq 3$ are treated separately.
If $n\geq 3$, let $\mathcal{G}$ be the linear graph of size $n$ ordered from node $v_1$ to $v_n$, and weighted with $w_{i,i+1} = 1 - a\one\{i=1\}$. Then, if $A=\{v_1,...,v_{\lceil n/32\rceil}\}$ and $d=(1-1/16)n - 1$, we have $|A_d^c|\geq|A|$ and \Lemma{A_lb} implies:
\BEQ
\avgFun^A(\theta_{i,t}) - \avgFun^A(\theta^*) \geq \frac{n\alpha}{2}\left(\frac{\sqrt{\kappa_g} - 1}{\sqrt{\kappa_g} + 1}\right)^{2\left(1+\frac{t}{1+d\tau}\right)} \|\theta_{i,0} - \theta^*\|^2.
\EEQ
A simple calculation gives $\kappa_l = 1 + (\kappa_g-1)\frac{n}{2|A|}$, and thus $\kappa_g\geq \kappa_l/16$. Finally, if we take $W_a$ as the Laplacian of the weighted graph $\mathcal{G}$, a simple calculation gives that, if $a=0$, $\gamma(W_a) = \gamma_n$ and, if $a=1$, the network is disconnected and $\gamma(W_a)=0$. Thus, by continuity of the eigenvalues of a matrix, there exists a value $a\in[0,1]$ such that $\gamma(W_a) = \gamma$. Finally, by definition of $n$, one has $\gamma > \gamma_{n+1} \geq\frac{2}{(n+1)^2}$, and $d\geq\frac{15}{16}(\sqrt{\frac{2}{\gamma}} - 1) - 1 \geq\frac{1}{5\sqrt{\gamma}}$ when $\gamma\leq\gamma_3=\frac{1}{3}$.

For the case $n=2$, we consider the totally connected network of $3$ nodes, reweight only the edge $(v_1,v_3)$ by $a\in[0,1]$, and let $W_a$ be its Laplacian matrix. If $a=1$, then the network is totally connected and $\gamma(W_a) = 1$. If, on the contrary, $a=0$, then the network is a linear graph and $\gamma(W_a) = \gamma_3$. Thus, there exists a value $a\in[0,1]$ such that $\gamma(W_a) = \gamma$, and applying \Lemma{A_lb} with $A=\{v_1\}$ and $d=1$ returns the desired result, since then $\kappa_g\geq 2\kappa_l/3$ and $d=1\geq \frac{1}{\sqrt{3\gamma}}$.
\end{proof}

\subsection{Convergence rates of SSDA and MSDA}

\begin{proof}[Proof of Theorem 3]
Each step of the algorithm can be decomposed in first computing gradients, and then communicating these gradients across all neighborhoods. Thus, one step takes a time $1+\tau$. Moreover, the Hessian of the dual function $F^*(\lambda\sqrt{W})$ is
\BEQ
(\sqrt{W}\otimes I_d)\nabla^2 F^*(\lambda\sqrt{W})(\sqrt{W}\otimes I_d),
\EEQ
where $\otimes$ is the Kronecker product and $I_d$ is the identity matrix of size $d$. Also, note that, in Alg.(2), the current values $x_t$ and $y_t$ are always in the image of $\sqrt{W}\otimes I_d$ (\ie the set of matrices $x$ such that $x^\top\one = 0$). The condition number (in the image of $\sqrt{W}\otimes I_d$) can thus be upper bounded by $\frac{\kappa_l}{\gamma}$, and Nesterov's acceleration requires $\sqrt{\frac{\kappa_l}{\gamma}}$ steps to achieve any given precision \cite{bubeck2015convex}.
\end{proof}

\begin{proof}[Proof of Theorem 4]
First, since $P_K(W)$ is a gossip matrix, Theorem 3 implies the convergence of Alg.(3). In order to simplify the analysis, we multiply $W$ by $\frac{2}{(1+\gamma)\lambda_1(W)}$, so that the resulting gossip matrix has a spectrum in $[1-c_2^{-1},1+c_2^{-1}]$. Applying Theorem 6.2 in \cite{chebyshevAcc} with $\alpha=1-c_2^{-1}$, $\beta=1+c_2^{-1}$ and $\gamma=0$ implies that the minimum
\BEQ
\min_{p\in\mathbb{P}_K,p(0)=0} \max_{x\in[1-c_2^{-1},1+c_2^{-1}]}|p(t)-1|
\EEQ
is attained by $P_K(x) = 1 - \frac{T_K(c_2(1-x))}{T_K(c_2)}$. Finally, Corollary 6.3 of \cite{chebyshevAcc} leads to
\BEQ
\gamma(P_K(W)) \geq \frac{1 - 2\frac{c_1^K}{1+c_1^{2K}}}{1 + 2\frac{c_1^K}{1+c_1^{2K}}} = \left(\frac{1-c_1^K}{1+c_1^K}\right)^2,
\EEQ
where $c_1 = \frac{1-\sqrt{\gamma}}{1+\sqrt{\gamma}}$, and taking $K=\lfloor\frac{1}{\sqrt{\gamma}}\rfloor$ implies
\BEQ
\frac{1}{\sqrt{\gamma(P_K(W))}} \leq \frac{1+c_1^{\frac{1}{\sqrt{\gamma}}+1}}{1-c_1^{\frac{1}{\sqrt{\gamma}}+1}} \leq 2.
\EEQ
The time required to reach an $\varepsilon>0$ precision using Alg.(3) is thus upper bounded by\\$O\left((1+K\tau)\sqrt{\frac{\kappa_l}{\gamma(P_K(W))}}\ln(1/\varepsilon)\right) = O\left(\sqrt{\kappa_l}(1+\frac{1}{\sqrt{\gamma}}\tau)\ln(1/\varepsilon)\right)$.
\end{proof}

\section{Composite problems for machine learning}
When the local functions are of the form
\BEQ
f_i(\theta) = g_i(B_i \theta) + c\|\theta\|^2,
\EEQ
where $B_i \in \R^{m_i \times d}$ and $g_i$ is smooth and has  proximal operator which is easy to compute (and hence also $g_i^*$), an additional Lagrange multiplier $\nu$ can be used to make the Fenchel conjugate of $g_i$ appear in the dual optimization problem. More specifically, from the primal problem of Eq.~(12), one has, with $\rho>0$ an arbitrary parameter:
\BEAS
\inf_{\Theta\sqrt{W}=0} F(\Theta)
& = & \inf_{\Theta\sqrt{W}=0,\ \forall i, x_i=B_i\theta_i}
\frac{1}{n} \sum_{i=1}^n g_i(x_i) + c\|\theta_i\|_2^2
\\ 
& = & \inf_\Theta \sup_{\lambda,\nu}
\frac{1}{n} \sum_{i=1}^n\Big\{\nu_i^\top B_i \theta_i - g_i^*(\nu_i) 
+ c \|\theta_i\|_2^2 \Big\} +
\frac{\rho}{n}\tr(\lambda^\top \Theta \sqrt{W})
\\
& = & 
 \sup_{\nu \in \prod_{i=1}^n \! \R^{m_i},\ \lambda \in \R^{d \times n  } }
- \frac{1}{n} \sum_{i=1}^n
 g_i^*(\nu_i)
 -\frac{1}{4c n} \sum_{i=1}^n \|  B_i^\top \nu_i +  \rho\lambda \sqrt{W}_i  \|_2^2.
\EEAS
To maximize the dual problem, we can use (accelerated) proximal gradient, with the updates:
\BEAS
\nu_{i,t+1} & = & \inf_{ \nu \in \R^{m_i}} g_i^*(\nu)
+ \frac{1}{2\eta} \big\| \nu - \nu_{i,t}  + \frac{\eta}{2c} B_i ( B_i^\top \nu_{i,t}  + \rho \lambda_t \sqrt{W}_i ) \big\|_2^2 \\
\lambda_{t+1} & = & \lambda_t - \eta  \frac{\rho}{2cn} \sum_{i=1}^n ( B_i^\top \nu_{i,t}  + \rho \lambda_t \sqrt{W}_i )\sqrt{W}_i^\top.
\EEAS
We can rewrite all updates in terms of $z_t = \lambda_t \sqrt{W} \in \R^{d \times n}$, as
\BEAS
\nu_{i,t+1} & = & \inf_{ \nu \in \R^{m_i}} g_i^*(\nu)
+ \frac{1}{2\eta} \big\| \nu - \nu_{i,t} + \frac{\eta}{2c} B_i ( B_i^\top \nu_{i,t}  + \rho z_{i,t} ) \big\|_2^2 \\
z_{t+1} & = & z_t - \eta  \frac{ \rho}{2cn} \sum_{i=1}^n ( B_i^\top \nu_{i,t}  + \rho z_i )W_i^\top.
\EEAS
In order to compute the convergence rate of such an algorithm, if we assume that:
\BIT
\item each $g_i$ is $\mu$-smooth,
\item the largest singular value of each $B_i$ is less than $M$,
\EIT
then we simply need to compute the condition number of the quadratic function
$$
Q(\nu,\lambda) =  \frac{1}{2\mu } \sum_{i=1}^n \| \nu_i\|_2^2 +
 \frac{1}{4c } \sum_{i=1}^n \|  B_i^\top \nu_i +  \rho\lambda \sqrt{W}_i  \|_2^2.
$$
With the choice $\rho^2 = \frac{1}{\lambda_{\max}(W)} \big( \frac{c}{\mu} + M^2)$, it is lower bounded by
$
\big( 1 + \mu \frac{M^2}{c} \big) \frac{4}{\gamma}
$, which is a natural upper bound on $\kappa_l  / \gamma$. Thus this 
essentially leads to the same convergence rate than the non-composite case with the Nesterov and Chebyshev accelerations, \ie  $\sqrt{\kappa_l / \gamma}$.

The bound on the conditional number may be shown through the two inequalities:
\BEAS
Q(\nu,\lambda) & \leqslant &  \frac{1}{2\mu } \sum_{i=1}^n \| \nu_i\|^2
+
 \frac{1}{2c } \sum_{i=1}^n \|    \rho\lambda \sqrt{W}_i  \|_2^2
+  \frac{1}{2c } \sum_{i=1}^n \|  B_i^\top \nu_i   \|_2^2,
\\
Q(\nu,\lambda) & \geqslant &  \frac{1}{2\mu } \sum_{i=1}^n \| \nu_i\|^2
+
\frac{1}{1+\eta}  \frac{1}{4c }  \sum_{i=1}^n \|    \rho\lambda \sqrt{W}_i  \|_2^2
 - \frac{1}{\eta}  \frac{1}{4c } \sum_{i=1}^n \|  B_i^\top \nu_i   \|_2^2,
\EEAS
 with $\eta = M^2 \mu / c$.

\bibliographystyle{unsrt}
\bibliography{distributed_dual}

\end{document}